\documentclass[12pt]{amsart}

\usepackage[english]{babel}
\usepackage{mathrsfs,amssymb}
\usepackage{mathtools}
\usepackage[colorlinks, citecolor = blue]{hyperref}

\usepackage[shortlabels]{enumitem}
\setlist[itemize]{leftmargin=25pt}
\setlist[enumerate]{leftmargin=25pt}

\usepackage[T1]{fontenc}
\usepackage{lmodern}
\usepackage{tikz}

\definecolor{restartcolor}{RGB}{185,45,45}
\definecolor{continuationcolor}{RGB}{35,95,175}

\newtheorem{theorem}{Theorem}[section]
\newtheorem{lemma}[theorem]{Lemma}

\newtheorem{cor}[theorem]{Corollary}

\theoremstyle{definition}

\newtheorem{con}[theorem]{Conjecture}

\theoremstyle{remark}
\newtheorem{remark}[theorem]{Remark}

\numberwithin{equation}{section}
\usepackage[colorinlistoftodos,prependcaption,textsize=small]{todonotes}

\newcommand{\N}{\ensuremath{\mathbb{N}}}

\newcommand{\R}{\ensuremath{\mathbb{R}}}

\renewcommand{\P}{\ensuremath{\mathbb{P}}}

\newcommand{\mc}{\mathcal}
\newcommand{\ms}{\mathscr}

\let \e=\varepsilon
\let \d=\delta

\let \O=\Omega

\let \ga=\gamma

\allowdisplaybreaks

\begin{document}
\title[Dual boundedness of the maximal operator]
{Dual boundedness of the maximal operator on concave Banach function spaces}

\author[A.K. Lerner]{Andrei K. Lerner}
\address[A.K. Lerner]{Department of Mathematics,
Bar-Ilan University, 5290002 Ramat Gan, Israel}
\email{lernera@math.biu.ac.il}

\thanks{The author was supported by ISF grant no. 1035/21.}

\begin{abstract}
We prove a conjecture of Nieraeth asserting that if the Hardy--Littlewood maximal operator $M$ is bounded on an $s$-concave Banach function space $X$ for some $s\in (1,\infty)$, then $M$ is also bounded on the associate space $X'$. The proof is based on a new argument establishing the Fefferman--Stein inequality on $X$.
\end{abstract}


\maketitle

\section{Introduction}
Define the Hardy--Littlewood maximal operator on $\mathbb R^n$ by
$$
Mf(x):=\sup_{Q\ni x}\frac1{|Q|}\int_Q |f|,
$$
where the supremum is taken over all cubes $Q\subset\mathbb R^n$ containing $x$.

Given a Banach function space $X$ on $\mathbb R^n$, we denote by $X'$ its associate space. In several important particular settings, boundedness of $M$ on $X$ is known to imply its boundedness on $X'$. This phenomenon is somewhat surprising: although $M$ can be linearized, the adjoint of a linearization is a different operator, so the implication cannot be explained by a direct duality argument.

A basic example is provided by weighted Lebesgue spaces. Let
$X=L^p(w), 1<p<\infty.$ By Muckenhoupt's theorem \cite{M72}, $M$ is bounded on $L^p(w)$ if and only if $w\in A_p$. Since
$w\in A_p\Leftrightarrow w^{1-p'}\in A_{p'}$ and $\bigl(L^p(w)\bigr)'=L^{p'}(w^{1-p'})$, it follows that
$$M:L^p(w)\to L^p(w)\Rightarrow M:L^{p'}(w^{1-p'})\to L^{p'}(w^{1-p'}).$$
Thus, in this setting, the dual boundedness property follows directly from the $A_p$ characterization.

The same phenomenon also occurs in variable Lebesgue spaces. Suppose that
$X=L^{p(\cdot)}(\mathbb R^n),1<p_-\leq p_+<\infty$. Diening \cite{D05} proved that
$$M:L^{p(\cdot)}\to L^{p(\cdot)}\Rightarrow M:L^{p'(\cdot)}\to L^{p'(\cdot)}.$$
Since $\bigl(L^{p(\cdot)}\bigr)'=L^{p'(\cdot)}$, this is precisely the transfer of boundedness from a space to its associate.
The argument here, however, relies on a characterization of a different nature, since the natural variable-exponent analogue of the $A_p$ condition is not sufficient.

In both examples, the dual boundedness property is obtained from a characterization specific to the corresponding class of spaces. This does not explain whether the implication
\begin{equation}\label{mimp}
M: X\to X \quad\Rightarrow\quad M: X'\to X'
\end{equation}
can instead be deduced from intrinsic geometric properties of $X$.

There are several reasons why such a general principle would be useful. For example, boundedness of the maximal operator on both a space and its associate plays an essential role in extrapolation on Banach and quasi-Banach function spaces; see \cite{CU17,LoN24,N23}. In particular, Rubio de Francia extrapolation for Banach function spaces implies that if an operator is bounded on $L^p(w)$ for every $w\in A_p$, at some $p>1$, then it is also bounded on $X$, provided that $M$ is bounded on both $X$ and $X'$; see, for instance, \cite{CU17}. A principle of the form \eqref{mimp} would therefore simplify the hypotheses of such results and facilitate their application to concrete function spaces.

In \cite{L10}, the boundedness of $M$ on $X'$ was characterized in terms of the Fefferman--Stein inequality \cite{FS72} on $X$.
Define the sharp maximal operator by
$$f^{\#}(x):=\sup_{Q\ni x}\frac{1}{|Q|}\int_Q |f-f_Q|,\qquad f_Q:=\frac{1}{|Q|}\int_Q f.$$
Let $S_0(\mathbb R^n)$ denote the space of measurable functions $f$ on $\mathbb R^n$ such that
$$
|\{x\in\mathbb R^n:|f(x)|>\alpha\}|<\infty
\qquad
\text{for every }\alpha>0.
$$

\begin{theorem}[\cite{L10}]\label{fschar}
Assume that $M$ is bounded on a Banach function space $X$. Then $M$ is bounded on $X'$ if and only if there exists a constant $C>0$ such that
\begin{equation}\label{FS}
\|f\|_X\leq C\|f^{\#}\|_X
\end{equation}
for every $f\in S_0(\mathbb R^n)$.
\end{theorem}

Theorem~\ref{fschar} reformulates \eqref{mimp}, but does not by itself identify geometric conditions on $X$ under which the implication holds. Indeed, a direct proof of \eqref{FS} may be no easier than proving boundedness of $M$ on $X'$. For instance, in the variable-exponent setting, establishing
$M:L^{p(\cdot)}\to L^{p(\cdot)}\Rightarrow\eqref{FS}$
directly encounters essentially the same difficulties as Diening's dual boundedness theorem.

The following conjecture of Nieraeth \cite{N26} identifies $s$-concavity as the natural intrinsic condition under which \eqref{mimp} should hold.

\begin{con}[\cite{N26}]\label{conj:Nieraeth}
Let $1<s<\infty$, and let $X$ be an $s$-concave Banach function space over $\mathbb R^n$. If $M$ is bounded on $X$, then $M$ is bounded on $X'$.
\end{con}

The concavity assumption is essential, as shown by the elementary example $X=L^\infty$. On the other hand, weighted and variable Lebesgue spaces fall within the scope of the conjecture: $L^p(w)$ is $p$-concave, while $L^{p(\cdot)}$, under the assumption
$
1<p_-\leq p_+<\infty,
$
is $s$-concave for every $s\ge p_+$. Thus Conjecture~\ref{conj:Nieraeth} would recover both of the examples above without appealing to the Muckenhoupt condition in the weighted case or to Diening's theorem in the variable-exponent case.

The validity of Conjecture \ref{conj:Nieraeth} is far from evident: boundedness of the maximal operator on a Banach function space and on its associate is generally governed by rather different mechanisms. In fact, our first attempts were directed toward constructing a counterexample, as it seemed plausible that $s$-concavity alone might be too weak. The failure of these attempts gradually suggested that the conjecture captures the correct structural principle.

Our main result proves the conjecture by establishing the Fefferman--Stein inequality on every $s$-concave Banach function space $X$ on which $M$ is bounded.

\begin{theorem}\label{thm:main}
Let $1<s<\infty$, and let $X$ be an $s$-concave Banach function space over~$\mathbb R^n$. Assume that $M$ is bounded on $X$. Then there exists a constant $C>0$ such that
$$
\|f\|_X\le C\|f^{\#}\|_X,\quad\bigl(f\in S_0(\mathbb R^n)\bigr).
$$
\end{theorem}

By Theorem~\ref{fschar}, this result immediately yields the following.

\begin{cor}\label{positive}
Conjecture~\ref{conj:Nieraeth} is true.
\end{cor}

Observe that we actually prove a slightly stronger statement than Theorem \ref{thm:main}: the Fefferman--Stein inequality holds with the dyadic sharp function $f^{\#}_{\mathscr D}$ on the right-hand side, where $\mathscr D$ is an arbitrary dyadic lattice. This stronger statement, in turn, yields the boundedness of $M$ on $X'$ by a rather simple argument. Thus, Conjecture \ref{conj:Nieraeth} can be proved without appealing to Theorem \ref{fschar}. To make the paper essentially self-contained, we outline this argument in the Appendix.

Several classical approaches to the Fefferman--Stein inequality are known. These include arguments based on good-$\lambda$ inequalities \cite{FS72} and on non-increasing rearrangements~\cite{BS79}. Such methods are not available in our setting, since $X$ need not be rearrangement-invariant. A modern approach based on sparse domination gives (see, for example, \cite{LN19})
$$
|f|
\lesssim
\sum_{Q\in\mathcal S}
\left(\frac1{|Q|}\int_Q|f-f_Q|\right)\chi_Q.
$$
However, the standard estimate of the resulting sparse operator naturally uses boundedness of $M$ on $X'$, which is precisely the conclusion one seeks to prove.

Our proof is entirely dyadic. We construct a stopping tree for the localized dyadic maximal operator and divide its cubes into restart and continuation cubes, according to the relation between the average over a cube and the corresponding local mean oscillation. We then organize the continuation cubes into coronas rooted at the restart cubes. A pruning estimate, in which the $s$-concavity of $X$ is used, produces geometric decay along these coronas and leads to the dyadic Fefferman--Stein inequality.

A particularly striking and immediate application of our main result concerns weighted variable Lebesgue spaces. Let $p:\mathbb R^n\to(1,\infty)$ be measurable and suppose that
$$
1<p_-:=\operatorname*{ess\,inf}_{\mathbb R^n}p \leq \operatorname*{ess\,sup}_{\mathbb R^n}p=:p_+<\infty.
$$
Given a weight $w$, define
$$\|f\|_{L_w^{p(\cdot)}}:=\|fw\|_{L^{p(\cdot)}}.$$
It was previously known \cite{L17} that boundedness of $M$ on $L_w^{p(\cdot)}$ implies boundedness on its associate $L_{w^{-1}}^{p'(\cdot)}$ under the additional assumption
$w(\cdot)^{p(\cdot)}\in A_\infty.$
This condition was used at several essential stages of the proof, and it was not even clear whether it could be weakened. Since $L_w^{p(\cdot)}$ is $p_+$-concave, the affirmative resolution of Conjecture \ref{conj:Nieraeth}  removes this assumption entirely and yields the following result.

\begin{cor}
\label{cor:weighted-variable}
Let $p:\mathbb R^n\to(1,\infty)$ be measurable with
$
1<p_-\leq p_+<\infty,
$
and let $w$ be an arbitrary weight. If the Hardy--Littlewood maximal operator is bounded on $L_w^{p(\cdot)}$, then it is bounded on $L_{w^{-1}}^{p'(\cdot)}$.
\end{cor}

The paper is organized as follows. Section 2 contains the necessary preliminary material. In Section 3 we prove Theorem \ref{thm:main}. Section~4 is devoted to applications of our main result. These include consequences for Musielak--Orlicz spaces and weighted Morrey and block spaces, and applications to extrapolation on Banach function spaces. Some of these consequences were already pointed out by Nieraeth \cite{N26} under the assumption that Conjecture~\ref{conj:Nieraeth} holds. Finally, Section 5 is an appendix in which we outline a direct proof of Conjecture~\ref{conj:Nieraeth}, avoiding the characterization given by Theorem \ref{fschar}.

\section{Preliminaries}
\subsection{Banach function spaces}
Let $L^0({\mathbb R}^n)$ denote the space of measurable functions on~${\mathbb R}^n$. A vector space $X\subseteq L^0({\mathbb R}^n)$ equipped with a norm $\|\cdot\|_X$
is called a Banach function space over ${\mathbb R}^n$ if it satisfies the following properties:
\begin{itemize}
\item {\it Ideal property}: If $f\in X$ and $g\in L^0({\mathbb R}^n)$ with $|g|\le |f|$, then $g\in X$ and $\|g\|_X\le \|f\|_{X}$.
\item {\it Fatou property}: If $0\le f_j\uparrow f$ for $\{f_j\}$ in $X$ and $\sup_j\|f_j\|_{X}<\infty$, then $f\in X$ and $\|f\|_{X}=\sup_j\|f_j\|_{X}$.
\item {\it Saturation property}: For every measurable set $E\subset {\mathbb R}^n$ of positive measure, there exists a measurable subset $F\subseteq E$ of positive measure
such that $\chi_F\in X$.
\end{itemize}

We refer to \cite{LN24}, where, in particular, one can find a discussion of
this choice of axioms.

It is well known (see, for example, \cite[Lemma~3.5]{LN24}) that the Fatou
property is equivalent to the following version of Fatou's lemma: for every
sequence of nonnegative  $f_j\in X$ with
$\liminf_{j\to\infty}\|f_j\|_X<\infty$, one has
$\liminf_{j\to\infty}f_j\in X$ and
$$
\left\|\liminf_{j\to\infty}f_j\right\|_X
\leq
\liminf_{j\to\infty}\|f_j\|_X.
$$

Given a Banach function space $X$, its associate space $X'$ consists of all
$f\in L^0(\R^n)$ such that
$$
\|f\|_{X'}:=\sup_{\|g\|_X\le 1}\int_{\R^n}|fg|<\infty.
$$

A Banach function space $X$ is called $s$-concave, $0<s<\infty$, if there
exists a constant $C_s\geq1$ such that, for every $N\in\N$ and every
$h_1,\dots,h_N\in X$,
$$
\left(\sum_{k=1}^N\|h_k\|_X^s\right)^{1/s}
\leq
C_s\left\|\left(\sum_{k=1}^N|h_k|^s\right)^{1/s}\right\|_X.
$$

\subsection{Dyadic lattices} Given a cube $Q\subset {\mathbb R}^n$, denote by ${\mathcal D}(Q)$ the set of all dyadic cubes with respect to $Q$, that is, the cubes obtained by repeated subdivision of $Q$ and each of its descendants into $2^n$ congruent subcubes.

Following \cite[Def. 2.1]{LN19}, a dyadic lattice ${\mathscr D}$ in ${\mathbb R}^n$ is any collection of cubes such that
\begin{enumerate}
\renewcommand{\labelenumi}{(\roman{enumi})}
\item Any child of $Q\in{\mathscr D}$ is in ${\mathscr D}$ as well, i.e. $\mc{D}(Q) \subseteq \ms{D}$.
\item
Any $Q',Q''\in {\mathscr D}$ have a common ancestor, i.e. there exists a $Q\in{\mathscr D}$ such that $Q',Q''\in {\mathcal D}(Q)$.
\item
For every compact set $K\subset {\mathbb R}^n$, there exists a cube $Q\in {\mathscr D}$ containing $K$.
\end{enumerate}

\begin{remark}\label{rem}
In particular, by (iii), one can choose an increasing chain of cubes $Q_j\in {\mathscr D}$ such that $Q_j\uparrow {\mathbb R}^n$.
\end{remark}

\subsection{The dyadic maximal operator}
Given a cube $Q_0$, consider the local dyadic maximal operator
$$M^d_{Q_0}f(x):=\sup_{\substack{Q\in{\mathcal D}(Q_0)\\x\in Q}}\frac{1}{|Q|}\int_Q|f|.$$

The following result is well known and we include its proof for the sake of completeness.

\begin{lemma}\label{dyad} For any nonnegative and integrable function $f$ on $Q_0$ and every $\tau>1$, there exists a family ${\mathcal S}\subset {\mathcal D}(Q_0)$
such that
\begin{enumerate}[(i)]
\item if $R,Q\in\mathcal S$ and $R\subsetneq Q$, then $f_R>\tau f_Q$;
\item for almost every $x\in Q_0$,
\begin{equation}\label{spb}
M^d_{Q_0}f(x)\le \tau\sum_{Q\in {\mathcal S}}f_Q\chi_{E_Q}(x),
\end{equation}
where the sets $E_Q\subset Q$ are pairwise disjoint.
\end{enumerate}
\end{lemma}

\begin{proof} We construct a family of stopping cubes
${\mathcal S}\subset\mathcal D(Q_0)$ recursively as follows.
Put $Q_0\in{\mathcal S}$. Once $Q\in{\mathcal S}$ has been selected, let
the stopping children of $Q$ (which we denote by $\text{ch}_{\mathcal S}(Q)$) be the collection of maximal dyadic cubes
$P\in {\mathcal D}(Q)$ such that $f_P>\tau f_Q$, and add all such cubes to~$\mathcal S$. Observe that property (i) holds trivially.

Since the stopping children are maximal, they are pairwise disjoint, and hence
$$\sum_{P\in \text{ch}_{\mathcal S}(Q)}|P|\le \frac{1}{\tau}|Q|.$$
From this the union of the cubes in the $k$-th stopping generation has measure at
most $\tau^{-k}|Q_0|$. It follows that the set of points belonging to an infinite descending chain
of stopping cubes has measure zero. Thus, setting for $Q\in {\mathcal S}$,
$$E_Q:=Q\setminus \bigcup_{P\in \text{ch}_{\mathcal S}(Q)}P,$$
we obtain that the sets $E_Q$ are pairwise disjoint and they cover $Q_0$ up to
a null set.

Let us prove (\ref{spb}). Assume that $x\in E_Q$, and let $x\in R\in {\mathcal D}(Q_0)$. Suppose first that $R\subset Q$.
If $f_R>\tau f_Q$, then $R$ is contained in one of the maximal cubes
$P\in\text{ch}_{\mathcal S}(Q)$. This would imply
$x\in P$, contrary to $x\in E_Q$. Therefore, $f_R\le\tau f_Q$.

Suppose now that $Q\subsetneq R$. Let
$$
Q=Q_m\subsetneq Q_{m-1}\subsetneq\cdots\subsetneq Q_0
$$
be the chain of stopping ancestors of $Q$. Choose $j$ so that $Q_j\subsetneq R\subseteq Q_{j-1}$.
Since $Q_j$ is maximal among the dyadic subcubes $P\subsetneq Q_{j-1}$ satisfying $f_P>\tau f_{Q_{j-1}}$, we have
$$
f_R\leq\tau f_{Q_{j-1}}<f_{Q_j}\leq f_Q.
$$
Unifying both considered cases implies
$$M^d_{Q_0}f(x)\le \tau f_Q,\quad x\in E_Q,$$
which proves (\ref{spb}).
\end{proof}

\subsection{An approximation result} Fix a dyadic lattice ${\mathscr D}$. Consider the dyadic sharp maximal operator
$$f_{\mathscr D}^{\#}(x):=\sup_{Q\ni x, Q\in {\mathscr D}}\frac{1}{|Q|}\int_Q|f-f_Q|.$$

\begin{lemma}\label{appr}
Let $f\in S_0({\mathbb R}^n)\cap L^\infty({\mathbb R}^n)$.  There are bounded compactly
supported functions $f_j$ such that for all $x\in {\mathbb R}^n$,
\begin{equation}\label{first}
\lim_{j\to \infty}f_j(x)=f(x),
\end{equation}
and
\begin{equation}\label{second}
(f_j)_{\mathscr D}^{\#}(x)\le 2f_{\mathscr D}^{\#}(x).
\end{equation}
\end{lemma}

\begin{proof} By Remark \ref{rem}, there exists an increasing sequence of cubes $Q_j\in {\mathscr D}$ such that $Q_j\uparrow~{\mathbb R}^n$. Define
$$f_j:=(f-f_{Q_j})\chi_{Q_j}.$$

Let us show that
\begin{equation}\label{Qj}
\lim_{j\to \infty}f_{Q_j}=0.
\end{equation}
Given $\d>0$,
$$|f_{Q_j}|\le\delta+\|f\|_\infty\frac{|\{x:|f(x)|>\delta\}|}{|Q_j|}.$$
Since $f\in S_0$, the numerator in the second term is finite, whereas
$|Q_j|\to\infty$.  Thus the second term tends to zero, and then
$\d\downarrow0$ proves (\ref{Qj}).
Since every point of ${\mathbb R}^n$ eventually belongs to some $Q_j$, (\ref{first}) follows from (\ref{Qj}).

To prove (\ref{second}), fix a cube $R\in {\mathscr D}$ and a point $x\in R$. Consider
$$\O(f_j;R):=\frac1{|R|}\int_R|f_j-(f_j)_R|.$$
If $R\cap Q_j=\emptyset$, then $\O(f_j;R)=0$. Otherwise, there are two cases. If $R\subset Q_j$, then on $R$ we have $f_j=f-f_{Q_j}$, and therefore
$$
\O(f_j;R)=\O(f;R).
$$
If $Q_j\subsetneq R$, then $(f_j)_R=0$, and hence
$$\O(f_j;R)=\frac{1}{|R|}\int_{Q_j}|f-f_{Q_j}|\le \frac{2}{|R|}\int_{Q_j}|f-f_{R}|\le 2\O(f;R).$$
Taking the supremum over all $R\ni x, R\in {\mathscr D}$ proves (\ref{second}).
\end{proof}

\section{Proof of Theorem \ref{thm:main}}
By the standard convexity--concavity renorming theorem, see, e.g.,
\cite[Theorem 1.d.8]{LT79}, we may replace the norm of $X$ by an equivalent
lattice norm for which the $s$-concavity constant is one.  Boundedness of
$M$ and the desired conclusion are preserved under an equivalent norm.
Thus, throughout the proof, we assume
$$
\big(\|h_1\|_X^s+\|h_2\|_X^s\big)^{1/s}\le \left\|\big(|h_1|^s+|h_2|^s\big)^{1/s}\right\|_X.
$$

Fix a dyadic lattice ${\mathscr D}$, and define the dyadic maximal operator
$$M^{\mathscr D}f(x):=\sup_{Q\ni x, Q\in {\mathscr D}}\frac{1}{|Q|}\int_Q|f|.$$
Denote $A:=\|M^{\mathscr D}\|_{X\to X}$. Then $1\le A<\infty$.

The following statement is an important ingredient of the proof, and the only place where the $s$-concavity of $X$ is used.

\begin{lemma}\label{lem:pruning}
Let $\mathcal F\subset {\mathscr D}$, and let $0<\e<1$. To every
$Q\in\mathcal F$ associate a number $c_Q>0$ and a measurable set
$B_Q\subset Q$ such that $|B_Q|\le\e|Q|.$
Set
$$U:=\sum_{Q\in\mathcal F}c_Q\chi_Q,\qquad V:=\sum_{Q\in\mathcal F}c_Q\chi_{B_Q}.$$
Assume that the coefficients $c_Q$ are $L$-lacunary along nested cubes:
$$Q'\subsetneq Q \Rightarrow c_{Q'}\ge Lc_Q,\quad L>1.$$
Then
$$\|V\|_X\le\rho\|U\|_X,$$
where
$$\rho:=\left[1-\left( \frac{(1-\varepsilon)(L-1)}{LA}\right)^s\right]^{1/s}<1.$$
\end{lemma}

\begin{proof} Put
$$
W:=U-V=\sum_{Q\in\mathcal F}c_Q\chi_{Q\setminus B_Q}.
$$
Fix $Q\in\mathcal F$.  For every $x\in Q$,
$$
M^{\mathscr D}W(x)\ge\frac1{|Q|}\int_Q W\ge (1-\e)c_Q.
$$
Consequently,
\begin{equation}\label{in1}
M^{\mathscr D}W\ge(1-\e)\sup_{Q\in\mathcal F}c_Q\chi_Q.
\end{equation}

At a fixed point $x$, all dyadic cubes from $\mathcal F$ containing
$x$ form a chain, and, by the assumption, their coefficients grow geometrically as
the cubes decrease. Therefore
\begin{equation}\label{in2}
U(x)\le\frac{L}{L-1}\sup_{Q\in\mathcal F}c_Q\chi_Q(x).
\end{equation}
Indeed, if
$Q_0\supsetneq Q_1\supsetneq\cdots\supsetneq Q_m$ are the cubes
containing $x$, then
$$
c_{Q_k}\le L^{-(m-k)}c_{Q_m},
$$
and hence
$$\sum_{k=0}^m c_{Q_k}\le c_{Q_m}\sum_{r=0}^m L^{-r}\le\frac{L}{L-1}c_{Q_m},$$
from which (\ref{in2}) follows.

By (\ref{in1}) and (\ref{in2}),
$$M^{\mathscr D}W\ge\eta U,\qquad \eta:=\frac{(1-\varepsilon)(L-1)}{L}.$$
Thus
\begin{equation}\label{in3}
\|W\|_X\ge\frac{\eta}{A}\|U\|_X.
\end{equation}

Since $U=W+V$ and $W,V\ge0$, the normalized $s$-concavity gives
$$\bigl(\|W\|_X^s+\|V\|_X^s\bigr)^{1/s}\le\left\|\bigl(W^s+V^s\bigr)^{1/s}\right\|_X \le\|W+V\|_X=\|U\|_X.$$
Combining this estimate with (\ref{in3}) completes the proof.
\end{proof}

\begin{proof}[Proof of Theorem \ref{thm:main}] Let $Q_0\in {\mathscr D}$. Define the local dyadic sharp function
$$f^{\#,d}_{Q_0}(x)=\sup_{\substack{Q\in{\mathcal D}(Q_0)\\x\in Q}}\frac{1}{|Q|}\int_Q|f-f_Q|.$$
We will prove that there is a constant $C>0$ such that for any non-negative $f\in L(Q_0)$,
\begin{equation}\label{local}
\|M^d_{Q_0}f\|_{X}\le C\left(\|f^{\#,d}_{Q_0}\|_{X}+f_{Q_0}\|\chi_{Q_0}\|_{X}\right).
\end{equation}

Let us first show that, once (\ref{local}) is established, the theorem follows easily. Indeed, assume first that
$f\ge0$ is bounded and compactly supported. Take a cube $Q_0\in {\mathscr D}$ such that $\text{supp}\,f\subset Q_0$
and $|\text{supp}\,f|\le \frac{1}{2}|Q_0|$. Then, for every $x\in Q_0$,
$$f_{Q_0}\le 2\frac{1}{|Q_0|}\int_{Q_0}|f-f_{Q_0}|\le 2f_{\mathscr D}^{\#}(x).$$
Since $f\le M^d_{Q_0}f$ almost everywhere, we obtain from (\ref{local}) that
\begin{equation}\label{ofs}
\|f\|_{X}\lesssim \|f_{\mathscr D}^{\#}\|_{X}\le \|f^{\#}\|_{X}.
\end{equation}

Next, the assumption that $f\ge 0$ is easily removed since $|f|_{\mathscr D}^{\#}\le 2f_{\mathscr D}^{\#}$. Further, (\ref{ofs})
is extended to any bounded function from $S_0({\mathbb R}^n)$ by Lemma \ref{appr} and the Fatou property of~$X$. Finally, in order
to extend (\ref{ofs}) to any $f\in S_0({\mathbb R}^n)$, define $f_N:=\min(f,N)$. Then $f_N\in S_0\cap L^{\infty}$. Using that the mapping $t\mapsto\min(t,N)$ is
$1$-Lipschitz, we have $(f_N)_{\mathscr D}^{\#}\le 2f_{\mathscr D}^{\#}$. Hence, applying first (\ref{ofs}) to $f_N$
and then letting $N\to \infty$, the Fatou property of $X$ implies (\ref{ofs}) for any $f\in S_0({\mathbb R}^n)$. This completes the proof.

It remains to establish (\ref{local}).
Fix numbers $\tau>1$ and $\ga>0$ which will be chosen later. Let ${\mathcal S}$ be a family from Lemma \ref{dyad}.

If $P\in\text{ch}_{\mathcal S}(Q)$ we call $P$ a \emph{continuation child} when
$$
\inf_{x\in P} f^{\#,d}_{Q_0}(x)\le \ga f_Q,
$$
and a \emph{restart child} otherwise.

Let $\widehat P$ be the dyadic parent of $P$ in ${\mathcal D}(Q_0)$. By maximality of $P$,
$f_{\widehat P}\le\tau f_Q$. Hence
$$f_P \le f_{\widehat P}+\frac1{|P|}\int_P|f-f_{\widehat P}|\le \tau f_Q+2^n\inf_Pf^{\#,d}_{Q_0}.$$
From this, if $P$ is a continuation child, then
\begin{equation}\label{contpr}
f_P\le (\tau+2^n\gamma)f_Q,
\end{equation}
and, if $P$ is a restart child, then
\begin{equation}\label{restpr}
f_P\le \left(\frac{\tau}{\ga}+2^n\right)\inf_Pf^{\#,d}_{Q_0}.
\end{equation}

Denote
$$
B_Q:=\bigcup_{\substack{P\in\text{ch}_{\mathcal S}(Q)\\P\text{ continuation}}}P,
$$
and let us show that
\begin{equation}\label{BQ}
|B_Q|\le \frac{\ga}{\tau-1}|Q|.
\end{equation}
Indeed, if $B_Q=\emptyset$, this is trivial. Otherwise, for every $P\in\text{ch}_{\mathcal S}(Q)$,
$$(\tau-1)f_Q<f_P-f_Q\le \frac{1}{|P|}\int_P|f-f_Q|.$$
Therefore,
$$(\tau-1)f_Q|B_Q|\le \int_Q|f-f_Q|\le |Q|\inf_Qf^{\#,d}_{Q_0}\le \ga|Q|f_Q,$$
which proves (\ref{BQ}).

Using the notion of continuation and restart cubes, we perform a corona decomposition of ${\mathcal S}$ as follows.
A cube is called a \emph{corona root} if it is either $Q_0$ or a
restart cube.  Denote the family of roots by ${\mathcal R}$.

For each restart cube $R$, let $\mathcal C(R)$ consist of $R$ together with all cubes $Q\in\mathcal S, Q\subset R$ that can be connected to $R$ by a chain of continuation children. Equivalently, starting from $R$, one follows only continuation children and stops whenever a restart cube is encountered.
The coronas ${\mathcal C}(R)$ form a disjoint partition of the stopping family ${\mathcal S}$:
$$
{\mathcal S}=\bigcup_{R\in{\mathcal R}}{\mathcal C}(R).
$$

\begin{figure}[ht]
\centering
\begin{tikzpicture}[
    x=0.34cm,
    y=0.95cm,
    interval/.style={line width=2.3pt,line cap=round},
    restart/.style={interval,draw=restartcolor},
    continuation/.style={interval,draw=continuationcolor},
    lab/.style={font=\small}
]

% The selected intervals all belong to one dyadic grid, but they do not
% exhaust their parent intervals and may have different dyadic lengths.

% Root.
\draw[restart] (0,0) -- (32,0)
    node[midway,above=3pt,lab,text=restartcolor] {$R$};

% First stopping generation: selected dyadic subintervals with a gap.
\draw[continuation] (0,-0.82) -- (8,-0.82)
    node[midway,above=3pt,lab,text=continuationcolor] {$Q_1$};
\draw[restart] (16,-0.82) -- (24,-0.82)
    node[midway,above=3pt,lab,text=restartcolor] {$R'$};

% Second stopping generation: different dyadic lengths and further gaps.
\draw[continuation] (0,-1.64) -- (2,-1.64)
    node[midway,above=3pt,lab,text=continuationcolor] {$Q_2$};
\draw[restart] (4,-1.64) -- (8,-1.64)
    node[midway,above=3pt,lab,text=restartcolor] {$R''$};
\draw[continuation] (16,-1.64) -- (20,-1.64)
    node[midway,above=3pt,lab,text=continuationcolor] {$Q_3$};
\draw[continuation] (22,-1.64) -- (24,-1.64)
    node[midway,above=3pt,lab,text=continuationcolor] {$Q_4$};

% Third stopping generation.
\draw[continuation] (0,-2.46) -- (1,-2.46)
    node[midway,above=3pt,lab,text=continuationcolor] {$Q_5$};
\draw[continuation] (16,-2.46) -- (17,-2.46)
    node[midway,above=3pt,lab,text=continuationcolor] {$Q_6$};
\draw[restart] (18,-2.46) -- (20,-2.46)
    node[midway,above=3pt,lab,text=restartcolor] {$R'''$};

% Legend.
\draw[restart] (4.0,-3.18) -- (6.3,-3.18);
\node[lab,anchor=west] at (6.7,-3.18) {restart cube};

\draw[continuation] (17.0,-3.18) -- (19.3,-3.18);
\node[lab,anchor=west] at (19.7,-3.18) {continuation cube};

\end{tikzpicture}
\caption{A stopping tree with restart cubes shown in red and continuation cubes shown in blue. The corona rooted at $R$ is
$\mathcal C(R)=\{R,Q_1,Q_2,Q_5\}.$
A restart cube does not belong to $\mathcal C(R)$ unless it is the root $R$. Thus the branches beginning at $R'$, $R''$, and $R'''$ are excluded from $\mathcal C(R)$, together with all cubes lying below them.}
\end{figure}
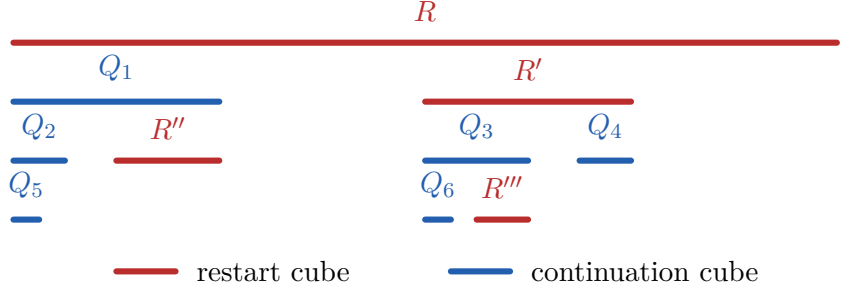

Write
$$
 \mathcal C(R)=\bigcup_{m\ge0}\mathcal C_m(R),
$$
where $\mathcal C_0(R)=\{R\}$ and $\mathcal C_m(R)$ is the $m$-th
continuation generation below $R$.

Let $\nu(Q)$ be the stopping generation of $Q$, so that $\nu(Q_0)=0$
and every stopping child increases $\nu$ by one. For $K\in {\mathbb N}$, which will be chosen later, define
$${\mathcal R}_j:=\{R\in{\mathcal R}:\nu(R)\equiv j\pmod K\},
 \qquad j=0,\dots,K-1.
$$
Then we have
$$
{\mathcal S}=\bigcup_{j=0}^{K-1}\bigcup_{R\in {\mathcal R}_j}\bigcup_{m\ge 0}\mathcal C_m(R).
$$

For $j=0,\dots,K-1$ and $m\ge0$, define
$$F_{j,m}:=\sum_{R\in {\mathcal R}_j}\sum_{Q\in\mathcal C_m(R)}f_Q\chi_{E_Q},\quad U_{j,m}:=\sum_{R\in {\mathcal R}_j}\sum_{Q\in\mathcal C_m(R)}f_R\chi_{Q}.$$
We have
\begin{equation}\label{split}
\sum_{Q\in{\mathcal S}}f_Q\chi_{E_Q}=\sum_{j=0}^{K-1}\sum_{m\ge0}F_{j,m}.
\end{equation}

If $Q\in \mathcal C_m(R)$, then applying (\ref{contpr}) subsequently $m$ times yields
$$f_Q\le (\tau+2^n\gamma)^mf_R.$$
Hence
\begin{equation}\label{Fj}
F_{j,m}\le (\tau+2^n\gamma)^mU_{j,m}.
\end{equation}

Now, in order to handle $U_{j,m}$ we will use Lemma \ref{lem:pruning}. First observe that, by the definition of the continuation generations,
$$U_{j,m+1}=\sum_{R\in {\mathcal R}_j}\sum_{Q\in\mathcal C_m(R)}f_R\chi_{B_Q}.$$
Indeed, $B_Q$ is exactly the disjoint union of the continuation children of $Q$.
We apply Lemma \ref{lem:pruning} to the family
$$\mathcal F_{j,m}:=\bigcup_{R\in{\mathcal R}_j}\mathcal C_m(R),$$
assigning to $Q\in\mathcal C_m(R)$ the coefficient $c_Q=f_R$. Let us check the lacunarity condition.
Suppose that $Q',Q\in {\mathcal F}_{j,m}$ and $Q'\subsetneq~Q$.
Observe that two strictly nested cubes cannot belong to the same family $\mathcal C_m(R)$. Therefore, there exist $R,R'\in {\mathcal R}_j$ such that $R\not=R'$ and $Q'\in {\mathcal C}(R'), Q\in {\mathcal C}(R)$.
It follows from this that $R'\subsetneq R$. Since both $R, R'\in {\mathcal R}_j$,
$$\nu(R')-\nu(R)\ge K,$$
which implies
$$c_{Q'}=f_{R'}>\tau^Kf_R=\tau^Kc_Q.$$
Therefore, by Lemma \ref{lem:pruning} along with (\ref{BQ}),
$$\|U_{j,m+1}\|_X\le\rho\|U_{j,m}\|_X,$$
where
$$\rho=\left[1-\left( \frac{(1-\e)(L-1)}{LA}\right)^s\right]^{1/s}\quad\Big(\e:=\frac{\ga}{\tau-1}, L:=\tau^K\Big).$$

By iteration,
$$\|U_{j,m}\|_X\le\rho^m\|U_{j,0}\|_X,$$
which, along with (\ref{Fj}), implies
\begin{equation}\label{Fjm}
\|F_{j,m}\|_X \le\big((\tau+2^n\gamma)\rho\big)^m\|U_{j,0}\|_X.
\end{equation}

Let us now fix the parameters $\tau,\ga$ and $K$. Denote
$$\rho_0:=\big(1-A^{-s}\big)^{1/s}<1.$$
We have $\rho\to \rho_0$ as $\e\to 0$ and $L\to \infty$. Hence, taking $\tau>1$ and sufficiently close to $1$, then $\ga>0$ small enough and $K$ large enough, we obtain that
$\nu:=(\tau+2^n\gamma)\rho<1$.

Combining Lemma \ref{dyad} with (\ref{split}) and (\ref{Fjm}), we obtain
\begin{equation}\label{estM}
\|M^d_{Q_0}f\|_{X}\le \tau\sum_{j=0}^{K-1}\sum_{m\ge 0}\|F_{j,m}\|_{X}\le\frac{\tau}{1-\nu}\sum_{j=0}^{K-1}\|U_{j,0}\|_{X}.
\end{equation}

It remains to estimate the root function
$$
U_{j,0}=\sum_{R\in{\mathcal R}_j}f_R\chi_R.
$$
The nested coefficients in this sum are $L$-lacunary. Therefore, by (\ref{in2}),
$$
U_{j,0}(x)\le\frac{L}{L-1}\sup_{R\in{\mathcal R}_j}f_R\chi_R(x).
$$
Every root $R\ne Q_0$ is a restart cube.  Hence, by (\ref{restpr}), for every $x\in R$,
$$
f_R\le \left(\frac{\tau}{\ga}+2^n\right)f^{\#,d}_{Q_0}(x).
$$
Consequently, for every $j=1,\dots, K-1$,
$$U_{j,0}\le\frac{L}{L-1}\left(\frac{\tau}{\ga}+2^n\right)f^{\#,d}_{Q_0},$$
and, for $j=0$,
$$U_{0,0}=f_{Q_0}\chi_{Q_0}+\sum_{R\in{\mathcal R}_j,R\not=Q_0}f_R\chi_R\le f_{Q_0}\chi_{Q_0}+\frac{L}{L-1}\left(\frac{\tau}{\ga}+2^n\right)f^{\#,d}_{Q_0}.$$
From this,
$$
\sum_{j=0}^{K-1}\|U_{j,0}\|_{X}\le f_{Q_0}\|\chi_{Q_0}\|_{X}+K\frac{L}{L-1}\left(\frac{\tau}{\ga}+2^n\right)\|f^{\#,d}_{Q_0}\|_{X},
$$
which, along with (\ref{estM}), proves (\ref{local}). Therefore the proof is complete.
\end{proof}

\section{Some applications}
There is an extensive literature in which the simultaneous boundedness of the Hardy--Littlewood maximal operator on a Banach function space $X$ and on its associate $X'$ appears as a natural structural assumption. Representative examples include Rubio de Francia extrapolation on Banach function spaces \cite{CU17}, extrapolation of compactness \cite{LoN24}, compactness characterizations of commutators on ball Banach function spaces \cite{TYYZ23}, the theory of unconditional wavelet bases \cite{K21}.

Theorem \ref{thm:main} shows that, when $X$ is $s$-concave, the boundedness of $M$ on $X'$ follows automatically from its boundedness on $X$. Consequently, in each of the results mentioned above, the two maximal-operator assumptions can be replaced by the single condition $M: X\to X$, provided, of course, that the remaining hypotheses of the corresponding result are satisfied. We record below several concrete instances of this principle. We begin with two applications to Musielak--Orlicz and weighted Morrey spaces; see \cite[Section~5]{N26} for the relevant definitions and a more thorough discussion.

\subsection{Musielak--Orlicz spaces}
Let $\varphi$ be a generalized $\Phi$-function and let
$\varphi^*$ be its complementary function.  Recall that
$L^{\varphi(\cdot)}$ is $s$-concave for some $s<\infty$ if and
only if $\varphi$ satisfies the $\Delta_2$-condition; see
\cite[Theorem~5.10]{N26}. Since
$$
\bigl(L^{\varphi(\cdot)}\bigr)'=L^{\varphi^*(\cdot)}
$$
up to equivalence of norms, Theorem \ref{thm:main} implies that
$$
M:L^{\varphi(\cdot)}\to L^{\varphi(\cdot)}\quad\Rightarrow\quad M:L^{\varphi^*(\cdot)}\to L^{\varphi^*(\cdot)}
$$
whenever $\varphi$ satisfies the $\Delta_2$-condition.

In particular, if both $\varphi$ and $\varphi^*$ satisfy the
$\Delta_2$-condition, then
$$
M:L^{\varphi(\cdot)}\to L^{\varphi(\cdot)}\quad\Leftrightarrow\quad M:L^{\varphi^*(\cdot)}\to L^{\varphi^*(\cdot)}.
$$
This proves \cite[Conjecture 5.11]{N26}, which is equivalent to \cite[Question 4.3.7]{HH19}.

\subsection{Weighted Morrey and block spaces}
Let $1<p\leq q<\infty$, and let $M_w^{p,q}$ and
$B_{w^{-1}}^{p',q'}$ denote, respectively, the weighted Morrey and
block spaces of Samko type.  They are related by the K\"othe duality
$$\bigl(M_w^{p,q}\bigr)'=B_{w^{-1}}^{p',q'}.$$
Although $M_w^{p,q}$ is not $s$-concave for any finite $s$ when
$p<q$, the block space $B_{w^{-1}}^{p',q'}$ is $p'$-concave.
Applying Theorem~\ref{thm:main} to the block space gives
$$
M:B_{w^{-1}}^{p',q'}\to B_{w^{-1}}^{p',q'}
\Rightarrow
M:M_w^{p,q}\to M_w^{p,q}.
$$
This proves \cite[Conjecture 5.17]{N26}.  It is worth
pointing out that the converse implication is false in general when
$p<q$, see \cite[Example 5.18]{N26}.

\subsection{Rubio de Francia extrapolation}
Let $\mathcal F$ be a family of pairs of nonnegative measurable
functions, and suppose that, for some $1\leq p_0<\infty$,
\begin{equation}
\label{rdfcond}
\int_{\mathbb R^n} f^{p_0}w
\lesssim
\int_{\mathbb R^n} g^{p_0}w
\end{equation}
for every $w\in A_{p_0}$ and every $(f,g)\in\mathcal F$, where
the implicit constant may depend on $[w]_{A_{p_0}}$, but not on
the particular weight $w$ or the pair $(f,g)$.
The extrapolation theorem for Banach function spaces
\cite[Theorem~10.1]{CU17} asserts that
$$
\|f\|_X\lesssim \|g\|_X,
\qquad (f,g)\in\mathcal F,
$$
provided that the Hardy--Littlewood maximal operator is bounded on
both $X$ and $X'$.

If $X$ is $s$-concave, Theorem~\ref{thm:main} shows that the
assumption on $X'$ is automatic. Hence, in this setting, the two
maximal-operator assumptions in the Banach function space
extrapolation theorem reduce to the single condition $M:X\to X$.

In particular, fix $1<q<\infty$, and let $\mathcal F_q$ be the
family of pairs
$$
\left(
\left(\sum_j (Mf_j)^q\right)^{1/q},
\left(\sum_j |f_j|^q\right)^{1/q}
\right),
$$
where the sums are finite. By the weighted vector-valued maximal
inequality \cite[Theorem~6.1]{CU17}, the family $\mathcal F_q$
satisfies \eqref{rdfcond} for every $1<p_0<\infty$. Therefore, if
$M$ is bounded on an $s$-concave Banach function space $X$, then
$$
\left\|
\left(\sum_j (Mf_j)^q\right)^{1/q}
\right\|_X
\lesssim
\left\|
\left(\sum_j |f_j|^q\right)^{1/q}
\right\|_X.
$$
Thus, on an $s$-concave Banach function space, scalar boundedness
of the maximal operator automatically upgrades to its
$\ell^q$-valued boundedness for every $1<q<\infty$.

\subsection{Extrapolation of compactness}
A similar consequence concerns extrapolation of compactness.  Let $T$
be a linear operator which is bounded on $L^{p_0}(v)$ for some
$1<p_0<\infty$ and every $v\in A_{p_0}$, and suppose that $T$ is
compact on $L^{p_1}(v_1)$ for some $1<p_1<\infty$ and some
$v_1\in A_{p_1}$.  By the extrapolation-of-compactness theorem of
Lorist and Nieraeth \cite{LoN24}, $T$ is compact
on every $r$-convex and $s$-concave Banach function space $X$,
$1<r<s<\infty$, for which $M$ is bounded on both $X$ and $X'$.
Theorem~\ref{thm:main} therefore yields the following simplified
implication:
$$
\begin{gathered}
        X\text{ is \(r\)-convex and \(s\)-concave for some }
        1<r<s<\infty,
\\
        M:X\to X
        \qquad\Longrightarrow\qquad
        T:X\to X\text{ is compact}.
\end{gathered}
$$

As a concrete instance, let $T$ be a Calder\'on--Zygmund operator
with a Dini-continuous kernel and let $b\in\mathrm{CMO}$.  Then, for
every $r$-convex and $s$-concave Banach function space~$X$,
$1<r<s<\infty$, such that $M:X\to X$, the commutator
$$
        [b,T]f=bTf-T(bf)
$$
is compact on $X$; see
\cite[Theorem~6.1]{LoN24}.

\section{Appendix}
Our proof of Conjecture~\ref{conj:Nieraeth} proceeds through the Fefferman--Stein inequality on $X$, using the characterization given by Theorem~\ref{fschar}, which was proved in \cite{L10}. In order to make the paper essentially self-contained, we now outline a direct proof of Conjecture~\ref{conj:Nieraeth} without appealing to that theorem.

Assume that $M$ is bounded on an $s$-concave Banach function space $X$. Let us show that $M$ is bounded on $X'$.
We proved in \eqref{ofs} that, for every dyadic lattice $\mathscr D$, the dyadic Fefferman--Stein inequality
\begin{equation}\label{dfs}
\|f\|_{X}\lesssim \|f_{\mathscr D}^{\#}\|_{X}
\end{equation}
holds for all $f\in S_0(\mathbb R^n)$.

By the three lattice theorem (see, for instance, \cite{LN19}), there exist $3^n$ dyadic lattices $\mathscr D_j$ such that
$$
Mf\lesssim \sum_{j=1}^{3^n}M^{\mathscr D_j}f.
$$
Consequently, it is enough to prove that $M^{\mathscr D}$ is bounded on $X'$, uniformly over all dyadic lattices $\mathscr D$.

Fix such a lattice $\mathscr D$. We may assume that $f\ge 0$. By the standard sparse linearization of the dyadic maximal operator (which can be proved as Lemma \ref{dyad}), there exists a sparse family
${\mathcal S}\subset{\mathscr D}$, with associated pairwise disjoint sets $E_Q\subset Q$, such that
$$
M^{\mathscr D}f
\lesssim
\sum_{Q\in\mathcal S}f_Q\chi_{E_Q}
=:T_{\mathcal S}f.
$$
It therefore suffices to prove that $T_{\mathcal S}$ is bounded on $X'$, uniformly in $\mathcal S$. By duality, this is equivalent to the uniform boundedness on $X$ of its adjoint
$$T_{\mathcal S}^{\star}g=\sum_{Q\in\mathcal S}\left(\frac1{|Q|}\int_{E_Q}g\right)\chi_Q.$$

We shall apply \eqref{dfs} to $T_{\mathcal S}^{\star}g$. Fix $R\in\mathscr D$. For $y\in R$, the dyadic nesting property gives
$$T_{\mathcal S}^{\star}g(y)=
\sum_{\substack{Q\in\mathcal S,Q\subsetneq R}}
\left(\frac1{|Q|}\int_{E_Q}g\right)\chi_Q(y)
+
c_R,
$$
where
$$
c_R
:=
\sum_{\substack{Q\in\mathcal S,R\subseteq Q}}
\frac1{|Q|}\int_{E_Q}g
$$
is independent of $y\in R$. Hence
\begin{eqnarray*}
\frac{1}{|R|}\int_R|T_{\mathcal S}^{\star}g-(T_{\mathcal S}^{\star}g)_R|&\le& \frac{2}{|R|}\int_R|T_{\mathcal S}^{\star}g-c_R|\\
&\le&\frac{2}{|R|}\|T_{\mathcal S}^{\star}(g\chi_R)\|_{L^1}\le \frac{2}{|R|}\int_Rg.
\end{eqnarray*}
In the last step we used the pairwise disjointness of the sets $E_Q$. From this
$$
(T_{\mathcal S}^{\star}g)^{\#}_{\mathscr D}(x)\le 2Mg(x).
$$
Therefore, by \eqref{dfs} and the boundedness of $M$ on $X$,
$$
\|T_{\mathcal S}^{\star}g\|_{X}\lesssim \|(T_{\mathcal S}^{\star}g)^{\#}_{\mathscr D}\|_{X}\lesssim \|Mg\|_{X}\lesssim \|g\|_{X},
$$
and this completes the proof.

\end{document}